\documentclass{amsart}

\usepackage{amsfonts}
\usepackage{amscd}
\usepackage{amssymb}
\usepackage{multirow}

\usepackage[hidelinks]{hyperref}
\newtheorem{theorem}{Theorem}[section]

\newtheorem{corollary}[theorem]{Corollary}
\newtheorem{proposition}[theorem]{Proposition}
\theoremstyle{remark}

\newtheorem{remark}[theorem]{Remark}
\numberwithin{equation}{section}

\begin{document}

\title[Descent methods for studying integer points on $y^{n}=F(x)G(x)$]{Descent methods for studying integer points on $y^{n}=f(x)g(x),\ n\ge 2$}
\address {K. A. Draziotis \\
Department of Informatics\\
Aristotle University of Thessaloniki \\
54124 Thessaloniki, Greece}
\email{drazioti@csd.auth.gr}

\maketitle

\begin{abstract}
We study the integer points on superelliptic and hyperelliptic curves of the form $y^n=f(x)g(x),$ $n\ge 2, \deg{f}+\deg{g}\ge 4.$
\end{abstract}

\section{Introduction}
The practical solution of integer points on hyperelliptic curves, $C:y^2=F(x)$ for $F\in{\mathbb{Z}}[x]$ $(\deg{F}\ge 3)$ not a square,  is a well studied problem. For instance if $F(x)$ is monic and quartic, we have very sharp (upper)  bounds for $\max\{|x|,|y|\}$ or explicit methods to solve them \cite{draz-quartic}. For $\deg{F}$ even, again we have polynomial bounds (but not very sharp), which in some cases we can just enumerate its integer points \cite{Tengely}. What is more, in this case we have specific algorithm to solve them, see \cite{laszlo2}. If $F(x)$ is of odd degree, the problem is harder. One reason for this, it is that the even degree can be treated with the so called Runge's method \cite[Appendix]{draz-quartic}. The same if $F(x)$ is not monic (even for even degree), the problem is more difficult. In general, the case where $F(x)$ is irreducible
and $\deg{F}\in\{3,4\}$ leads to an elliptic curve, in  which case we have a very good knowledge to practically solve it, for instance see \cite{tzanakis}. The case where $F$ is irreducible and $\deg{F}\ge 5,$ was studied in \cite{BMSST}, where two assumptions must be satisfied so that their method be successful. They need a rational point on $C$ and a basis for the jacobian of $C$ over the rationals. Another method for $y^2=F(x)$ and $\deg{F}$ odd is the so called {\it quadratic Chabauty} \cite{BBM}. In \cite{BBM} they combined Mordell-Weil sieve and $p-$adic approximations techniques to attack the problem. 

Now, let $F(x)$ be a reducible polynomial over the rationals. In this paper we address this case. The  {\it quadratic Chabauty}  method of \cite{BBM} is still valid. Now there is also another method, the so called {\it descent method}. Fot instance in \cite{CooGr}, they use covers of $C$ (in their paper they call such covers, heterogeneous spaces) in order to find rational points on curves of genus $2,$ where $F(x)$ is factorized in ${\mathbb{Q}}$ as $f(x)g(x).$ In \cite{draz1} we used such covers to study the integer points on families of elliptic curves $y^2=x(x^2\pm p)$ ($p$ prime). This method is used in the current work to study hyperelliptic equations. We study the integer points on the  hyperellipitc curve 
$$y^2=(x^3+x+1)(x^4+2x^3-3x^2+4x+4),$$
which it was also studied in  \cite[Section 9]{BBM} (using the {\it{quadratic Chabauty}} method). It was proved that the integer solutions $(x,y)$ with $y>0,$ are
$$\Sigma=\{(-1,2),(0,2),(-2,12),(3,62)\}.$$
It is easy to show that the set of integer solutions is indeed $\Sigma$ using descent method, without the heavy computations on jacobian varieties.

We provide a more general framework, where superelliptic curves ${\mathcal{M}}:y^p=f(x)g(x),$ also fits. For instance, in \cite[Section 8]{draz1} we tried this method in $y^3=x^4+kx.$ In the present paper we study $y^p=(\alpha x^p+\beta)g(x)$ ($p$ prime). Therefore, in part, we generalize the results of \cite[Section 8]{draz1}. 

In the bibliography, for superelliptic curves we have the following method (see \cite[Chapter IV, Section 4]{Schmidt}),  which we call it {\it{reduction to S-unit equations}}. Under some conditions it may be practical.  It is based on factorization over a number field. Say, $y^n=f(x),$ $n\ge 3,\ d = \deg{f}\ge 2$ and $\alpha, \beta$ are simple roots of $f(x).$  If $(x,y)$ is an integer point with $y\not=0,$ then, $x-\alpha=\kappa \xi^n$ where it is proved that $\kappa$ belongs to a finite set of integral elements of the number field $K={\mathbb{Q}}(\alpha),$ for the proof of this see \cite[Lemma 4B]{Schmidt}. Similar if $\beta$ is another simple root we get an equation of the form $x-\beta=\kappa'\xi'^n.$ If $d\ge 3$, then we end up with finitely many Thue equations $\kappa\xi^n-\kappa'\xi'^n=\beta-\alpha$ in ${\mathbb{Q}}(\alpha,\beta)$ with unknowns $\xi, \xi'.$ If it occurs $\alpha, \beta, \xi, \xi', \kappa, \kappa'$ be rational integers, we can solve the Thue equations over the integers.  

For further information on descent method see \cite{CooGr}, \cite[Section 2.3]{duq}, \cite[Section 1.3]{Gajovic}, \cite[Section 5.5]{Poonen} and their references.

\subsection{Our Contribution}
We study the following families of superelliptic curves
$y^p=(\alpha x^p + \beta)g(x)$ for $g\in{\mathbb{Z}}[x].$ We provide an implementation for solving such Diophantine equations. In Corollary \ref{cor:superelliptic} we bound the number of integer points on the curve $y^p=(\alpha x^p \pm 1)(\gamma x^m+\beta).$ 

Furthermore, we study hyperelliptic curves of the form $y^2=f(x)g(x),$ where $y^2=f(x)$ defines an elliptic curve. We provide an implementation and we present many examples. What is more, we study the following families $y^2=(x^4\pm 1)g(x).$ For instance, we prove that $y^2=(x^4+1)(x^n-1)$ has only the trivial integer solutions, for all $n\not\equiv 0\pmod{4}.$ \\
{\bf Roadmap}. In section \ref{sec:superelliptic} we study the curve $y^p=f(x)g(x),$ for $p$ prime. In section \ref{sec:hyperelliptic}, we study the hyperelliptic curve $y^2=f(x)g(x).$ In the final section, we summarize our results and provide some concluding remarks. Also, there is an appendix where we provide pseudocode for solving the Diophantine equations $y^2=(x^4- 1)g(x)$ and we provide tables for the solutions of the families $y^2=(x^4-1)g(x)$ for $g(x)\in\{x^2+kx+1,\ x^3+kx^2+1,\ x^3+kx+1\}.$
\section{The case $y^p=f(x)g(x),$ $p$ odd prime}\label{sec:superelliptic}
Let
\begin{equation}
C:y^p=f(x)g(x)
\end{equation}
where $g(x),f(x)$ do not have common roots over the algebraic closure ${\overline{\mathbb{Q}}}.$ Let $V\subset {\mathbb{A}}^4({\overline{\mathbb{Q}}})$\footnote{With ${\mathbb{A}}^n(K)$ we denote the $n-$th dimensional affine space over the field $K.$} be the algebraic curve defined by the equations 
\[
u^p=f(x), v^p=g(x), y=uv.
\]
We consider the projection \[
\Psi:V\rightarrow C,\ \Psi(x,y,u,v)=(x,y).
\]
$\Psi$ is a finite unramified map of degree $p.$ It is finite, since  it is a projection and also unramified, since ${\mathcal{Z}}(f,g)=\emptyset\footnote{With $\mathcal{Z}(f,g)$ we denote the set of common zeros of $f(x),g(x)$ in ${\overline{\mathbb{Q}}}$}.$ Indeed, if we fix $(x,y)$ then the system
\[
u^p=f(x), v^p=g(x), uv=y
\]
has exactly $p-$ solutions in ${\overline{\mathbb{Q}}}.$ To see this, it is enough to enumerate the distinct pairs $(\zeta_i,\zeta_j)$ ($\zeta_i^p=\zeta_j^p=1$), such that $\zeta_i\zeta_j=1.$ 
Now, from Chevalley-Weil Theorem \cite{Chev,DrazPoul1,DrazPoul2}, \cite[Section 4.2, p. 50]{Serre} we get finitely many number fields $K={\mathbb{Q}}(u,v)$ with degree at most $p$ over ${\mathbb{Q}}$. Since $p$ is a prime then $K={\mathbb{Q}}$ or is of degree exactly $p.$  This provides the following constraints in $(u,v).$ If we fix $(x,y)=(a,b)\in {\mathbb{Z}}^2$ then  necessarily, 
\begin{equation}\label{relation_u0_v0_}
f(a)=d_1X^p, g(a)=d_2Y^p,\ d_1d_2=Z^p,
\end{equation}
for some $X,Y,Z$ integers, $d_1,d_2$ $p-$free\footnote{$d_1,d_2$ do not have any prime power of the form $q^p$ in their primitive factorization} non zero integers, and  $i\in\{0,1,...,p-1\}.$
Assume that $|d_1|\not=1$ and let $\{p_1,...,p_\xi\}$ be the prime divisors of $d_1.$ Say, 
$d_1=p_1^{a_1}\cdots p_{\xi}^{a_{\xi}},$ where $1\le a_i\le p-1.$ Then, from the third equation of (\ref{relation_u0_v0_})  we get $d_2=p_1^{p-a_1}\cdots p_{\xi}^{p-a_{\xi}}.$
Note that ${\rm{rad}}(d_1)={\rm{rad}}(d_2),$ where ${\rm{rad}}(k)$ is the largest square free divisor of $k.$ We set $d={\rm{rad}}(d_1).$ If $d_1=1,$ then also $d_2=1$ and similar if $d_1=-1.$ So again $rad(d_1)=rad(d_2)=d.$ We shall prove the following Proposition.
\begin{proposition}\label{prop:superelliptic}
If $R(u,v)=res_{x}(u^p-f(x),v^p-g(x))\ (p\ge 3)$ and $c=R(0,0),$ then, $d|c,$ where $d=rad(d_1)=rad(d_2).$
\end{proposition}
\begin{proof}
We note that $R(u,v)=R_0(u^p,v^p),$ for some $R_0\in {\mathbb{Z}}[X,Y].$ Let $(a,b)$ be an integer point of $C:y^p=f(x)g(x).$
Let $(u_0,v_0)$ such that $u_0^p=f(a),\ v_0^p=g(a).$ 
Thus (as in (\ref{relation_u0_v0_})) we get,
$$u_0^p=d_1X^p,\ v_0^p=d_2Y^p,\ d_1d_2=Z^p.$$
But, $R(u_0,v_0)=0.$ Therefore, $R_0(u_0^p,v_0^p)=R_0(d_1X^p,d_2Y^p)=0.$ So we get $$R_0(d_1X^p,d_2Y^p)=dR_1(X^p,Y^p)+c=0,$$ for some $R_1\in{\mathbb{Z}}[X,Y]$ and integer $c=R(0,0)\not=0.$ 
We remark that $c=rex_x(f,g)$ (up to sign), so $c\not=0$ , since $f,g$ do not have common roots. We conclude that $d|c.$
\end{proof}
The previous is a basic result for what follows. More general results, similar to Proposition \ref{prop:superelliptic} were given in \cite[Lemma 4B]{Schmidt} and \cite[Lemma 2.2.1]{Bruin}.

Having the previous Proposition, for $(a,b)\in C({\mathbb{Z}}),$ we get an integer $z$ such that $dz^p=f(z),$ for some $d$ belonging to the set,
\begin{equation}\label{theset}
{\mathcal{S}}_{f,g,p}=\{k\in {\mathbb{Z}}\ (p-\text{free}):k|c, c=res_x(u^p-f,v^p-g)|_{(u=0,v=0)}\}.
\end{equation}
In general, the Diophantine equation $dz^p=f(x)$ is not easier than the initial $y^p=f(x)g(x)$, thus it may be hard to solve it. However, in some cases is more easy. For instance, if $C:y^p=(Ax^p+B)g(x),$ $A,B\in {\mathbb{Z}}-\{0\}$ and $g(x)$ does not have common roots with $Ax^p+B,$ then we reduce the study of $C$ to finitely many binomial Thue equations. Since Thue equations can be solved practically for moderate degree, for instance using Pari \cite{PARI2} or Magma \cite{Magma}, the same is true for $C.$

Also, we get the following result.
\begin{corollary}\label{cor:superelliptic}
Let $m,p$ be positive integers with $p\ge 3$ prime, $m\ge 2$ and $\alpha,\beta,\gamma$ non zero integers. Then, the number of integer solutions of the Diophantine equation 
$$y^p = (\alpha x^p \pm 1)(\gamma x^m+\beta)$$ is at most
$2\tau(c),$  where $c=-\alpha^m \beta^p \mp \gamma^p$ and $\tau(k)$ is the number of positive divisors of integer $k.$ 
\end{corollary}
\begin{proof}
We consider the Diophantine equation 
\begin{equation}\label{eq:at_most_one_sol}
y^p = (\alpha x^p \pm 1)(\gamma x^m+\beta),\ \text{where}\ p\ge 3,\ \text{and} \ m\ge 2.
\end{equation} 
Set $f(x)=\alpha x^p\pm 1,$ $g(x)=\gamma x^m+\beta,$ then
$c=R(0,0)=-\alpha^m \beta^p \mp \gamma^p,$ where
$R(u,v) = res_{x}(u^p-f(x),v^p-g(x)).$ From Proposition \ref{prop:superelliptic}, we have to solve equations of the form
\begin{equation}\label{eq:twist}
dz^p=\alpha x^p \pm 1, \ d|c\ (d\ \text{is}\ p-\text{free}).
\end{equation}
These are (binomial) Thue equations and they have been already studied. These equations have at most one integer solution, see \cite[Chapter 3, Lemma 3.3]{gaal}. So, overall the number of integer solutions of (\ref{eq:at_most_one_sol}) is at most $|\{d\in {\mathbb{Z}}:d|c,\ d\ \text{is} \ p\text{-free})\}|\le 2\tau(c).$ 
\end{proof} 
\subsection{The Algorithm}
Now, we easily get the following pseudocode\footnote{for an implementation in Sagemath see,\\ \url{https://github.com/drazioti/hyper_super_elliptic/blob/main/superelliptic.sage}}.\\\\
{\footnotesize{
\noindent
{\bf Input}:  $p$ odd prime, $f(x)=Ax^p+B$ and $g(x)\in{\mathbb{Z}}[x]$ ($f,g$ do not have common roots).
\ \\
{\bf Output}: The integer solutions of $y^p=f(x)g(x).$ \\\\
    \texttt{01.} \texttt{$L, M$ two empty lists}\\
    \texttt{02.} \texttt{$R\leftarrow rex_x(u^p-f(x),v^p-g(x))$}\\
	\texttt{03.} \texttt{$R_0\leftarrow R(0,0)$}\\
    \texttt{04.} \texttt{$PF\leftarrow \{d\in {\mathbb{Z}}:d|R_0,\ d\ \text{is\ p-free\ integer}\}$\\
	\texttt{05.} {\bf{for}} $d$ in $PF$\\
	\texttt{06.} \hspace{0.3cm}  \texttt{Solve\  Thue\ equation ${\mathcal{T}}_d:\pm dy^p-Ax^p=B$}\\
	\texttt{07.} \hspace{0.3cm} \texttt{append  $L$ with $x=a$ where $(x,y)=(a,b)\in {\mathcal{T}}_d({\mathbb{Z}})$}\\
	\texttt{08.} {\bf for} $z$ in $L$\\
	\texttt{09.} \hspace{0.3cm} {\bf if} $f(z)g(z)$ \texttt{is pth power}\\
	\texttt{10.} \hspace{0.6cm} Append list $M$ with $(z,\sqrt[p]{f(z)g(z)})$\\
	\texttt{11.}    return $M$
}}}
\subsection{Examples}
$({\bf i}).$ Let the curve 
$y^3=(x^3+691)(x^2-17).$ Executing the previous algorithm we get only one point $(13,76).$\\
$({\bf ii}).$ The curve 	 
$y^3=(x^3+625)(x+1)$ has the integer solutions $$\{(-10,15),(-1,0),(15,40)\}.$$
$({\bf iii}).$ The curve  $y^5=(x^5-724)(x+2)$  has the integer solutions $$\{(-2,0),(5,7)\}.$$
\section{The case $y^2=f(x)g(x)$ where $y^2=f(x)$ defines an elliptic curve}\label{sec:hyperelliptic} Let 
\begin{equation}
C:y^2=f(x)g(x),
\end{equation}
with $f(x)$ a cubic or quartic polynomial and $g(x),f(x)$ do not have common root over the algebraic closure 
${\overline{\mathbb{Q}}}.$ We remark that if $\deg{g}+\deg{f}$ is even and the leading coefficient of $f(x)g(x)$ is one, then $C$ is of {\it{Runge}} type and in this case there are uniform upper bounds (for instance see \cite[Introduction]{draz-quartic}). As in the superelliptic case, let $V\subset {\mathbb{A}}^4({\overline{\mathbb{Q}}})$ be the algebraic curve defined by the equations: 
\[
u^2=f(x), v^2=g(x), y=uv.
\]
We consider the projection \[
P:V\rightarrow C, P(x,y,u,v)=(x,y).
\]
$P$ is a finite unramified map of degree $2.$  Indeed, if we fix $(x,y)$ then the system
\[
u^2=f(x), v^2=g(x), uv=y
\]
has exactly two solutions in ${\overline{\mathbb{Q}}}.$
Now, from Chevalley-Weil Theorem we get that the number field ${\mathbb{Q}}(u,v)$ is at most of degree two. This provides the following constraints in integer solutions. If we fix $(x,y)=(a,b)\in C({\mathbb{Z}})$ then  necessarily, 
\begin{equation}\label{relation_u0_v0}
u_0^2=f(a)=dX^2, v_0^2=g(a)=dY^2,
\end{equation}
so,
$$(u_0,v_0)=(\sqrt{d}|X|,\sqrt{d}|Y|),$$ for some $X,Y,$ integers and $d$ square free, non zero integer.
\begin{proposition}
Set $R(u,v)=res_{x}(u^2-f(x),v^2-g(x))$ and $c=R(0,0).$
Then $d|c.$
\end{proposition}
\begin{proof}
First we remark that, $res_x(u^2-f(x),v^2-g(x))$ at $(u,v)=(0,0)$ equals to $R_0=\varepsilon\cdot res_x(f(x),g(x))$, $\varepsilon\in\{-1,1\}.$ 
Let $(a,b)$ such that $b^2=f(a)g(a).$ Let also $u_0^2=f(a), v_0^2=g(a).$ We note that $R(u,v)=R_0(u^2,v^2),$ for some $R_0\in {\mathbb{Z}}[X,Y].$ We set $u=u_0, v=v_0,$ as in (\ref{relation_u0_v0}), then $$R_0(u_0^2,v_0^2)=R_0(dX^2,dY^2)=dR_1(u_0^2,v_0^2)+c=0,$$ for some $R_1\in{\mathbb{Z}}[X,Y],$ thus, $d|c.$
\end{proof}
\begin{remark}
The set where $d$ belongs, is in fact ${\mathcal{S}}_{f,g,2}$ given in (\ref{theset}).
\end{remark}
\begin{corollary}
If $f(x)=(x-a_1)(x-a_2)(x-a_3)$ $(a_1,a_2,a_3,$ are distinct integers), then $d|g(a_1)g(a_2)g(a_3).$
\end{corollary}
\begin{proof}
This is immediate, since the resultant of two polynomials $f(x),g(x)$ can be expressed as 
$$res_x(f(x),g(x))=\prod_{i=1}^{\deg(f)} g(\lambda_i)\ (\text{up to sign}),$$
where $\{\lambda_i\}_i$ are the roots of $f(x).$
\end{proof}
Assume that $f(x)=x^3+Ax^2+Bx+C$. Let $(a,b)\in C({\mathbb{Z}})$, then there exists integer $b'$ such that $(a,b')$ belongs to the curve\footnote{In the bibliography $C_d'$ is also called {\it{heterogeneous space} or {\it{twist}}}.} 
\begin{equation}\label{C'}
C_d':dy^2=f(x).
\end{equation}
Let 
\[ 
\Psi_d :C_d'\rightarrow {\mathcal{E}}_d, \Psi_d(x,y)=(dx,d^2y)
\]
where,
$${\mathcal{E}}_d : y^2= x^3 + Adx^2 + Bd^2x + d^3C$$ 
and in the case where $f(x)$ has three integer roots $a_1,a_2,a_3,$ then 
$${\mathcal{E}}_d : y^2=(x-da_1)(x-da_2)(x-da_3).$$
 We have $\Psi_d(C_d'({\mathbb{Z}}))\subset {\mathcal{E}}_d({\mathbb{Z}})$, so in order to compute the integer points of $C'_d$ it is enough to compute the integer points of ${\mathcal{E}}_d.$
\begin{proposition}
Let $f(x)=x^3 + Ax^2 + Bx + C\in {\mathbb{Z}}[x]$ and $C:y^2 = f(x)g(x).$ 
We set $N=R(0,0),$ where $R(u,v)=res_x(u^2-f(x),v^2-g(x)).$
Let also,
$\Sigma_N=\{d\in {\mathbb{Z}}: d\ \ {\rm{squarefree}},\ d|N \},$ and 
$${\mathcal{M}}_{N}(f)=\bigg\{\frac{a}{d}: d\in \Sigma_N, {\rm{there \ is}}\ b\in{\mathbb{Z}}, {\rm{with}}\ d^2|b\ {\rm{and}}\ b^2=a^3 + Ada^2 + Bd^2a + d^3C \bigg\}.$$
Then, the set $C({\mathbb{Z}})=\{(A,B)\in{\mathbb{Z}}^2: A\in {\mathcal{M}}_N(f),\ B^2=f(A)g(A)\}$
\end{proposition}
This Proposition provides the reduction of the study of $y^2=f(x)g(x)$ to the elliptic curve 
$$y^2=a^3 + Ada^2 + Bd^2a + d^3C .$$
If $f(x)=(x-a_1)(x-a_2)(x-a_3),$ then we get the following.
\begin{corollary}
Let $C:y^2 = (x-a_1)(x-a_2)(x-a_3)g(x).$
We set 
$$N=g(a_1)g(a_2)g(a_3),\ \Sigma_N=\{d\in {\mathbb{Z}}: d\ \ {\rm{squarefree}},\ d|N \},$$ and 
$${\mathcal{M}}_{N}(f)=\bigg\{\frac{a}{d}: d\in \Sigma_N, {\rm{there \ is}}\ b\in{\mathbb{Z}}, {\rm{with}}\ d^2|b\ {\rm{and}}\ b^2=\prod_{i=1}^{3}(a-da_i) \bigg\}.$$
Then, the set $C({\mathbb{Z}})=\{(A,B)\in{\mathbb{Z}}^2: A\in {\mathcal{M}}_N(f),\ B^2=f(A)g(A)\}$
\end{corollary}
\subsection{The algorithm}
Now, we easily get the following pseudocode\footnote{for an implementation in Sagemath see\\ \url{https://github.com/drazioti/hyper_super_elliptic/blob/main/hyperelliptic.sage}}.\\\\
{\footnotesize{
\noindent
{\bf Input}:  $f(x)=x^3+Ax^2+Bx+C\in {\mathbb{Z}}[x]$ (with non zero discriminant) and $g(x)\in{\mathbb{Z}}[x]$ ($f,g$ do not have common roots).
\ \\
{\bf Output}: The integer solutions of $y^2=f(x)g(x)\ (y>0).$ \\\\
    \texttt{01.} \texttt{$L, M$ two empty lists}\\
	\texttt{02.} \texttt{$R_0\leftarrow res_x(f,g)$}\\
    \texttt{03.} \texttt{$SF\leftarrow \{d\in {\mathbb{Z}}:d|R_0,\ d\ \text{is\ square free}\}$\\
	\texttt{04.} {\bf for} $d$ in $SF$\\
	\texttt{05.} \hspace{0.3cm}  \texttt{Solve ${\mathcal{E}}_d:y^2=x^3+Adx^2+Bd^2x+Cd^3$}\\
	\texttt{06.} \hspace{0.3cm} \texttt{append  $L$ with $a$ where $(a,b)\in {\mathcal{E}}_d({\mathbb{Z}})$}\\
	\texttt{07.} {\bf for} $z$ in $L$\\
	\texttt{08.} \hspace{0.3cm} {\bf if} $f(z)g(z)$ \texttt{is square}\\
	\texttt{09.} \hspace{0.6cm} Append list $M$ with $(z,\sqrt{f(z)g(z)})$\\
	\texttt{10.}    return $M$
}}}
\begin{remark}
If $$C:Dy^2=f(x)g(x),\ D\in {\mathbb{Z}}-\{0\}$$ then we can reduce this case to the previous, by considering the curve
$$C':y'^2=f(x)G(x)$$
where $y'=Dy$ and $G(x)=Dg(x).$
\end{remark}
\begin{remark}
Let $\theta(N),$ where $N=R(0,0),$ be the number of square free divisors of $N.$ Then, the number of equations 
$\{{\mathcal{E}}_d\}_{d\in{\Sigma_N}}$ is $2\theta(N).$ But, $\theta(N)=2^{\omega(N)},$ where 
$ \omega(N)=|\{p:p\ \text{prime}, p|N\}|.$ So the number of equations 
${\mathcal{E}}_d$ is $2^{\omega(N)+1}.$
\end{remark}
\begin{remark}
In magma \cite{Magma}, the command \texttt{SIntegralLjunggrenPoints([D,A,B,C],
[])}, provides the integral points 
on the curve $C : Dy^2 = Ax^4 + Bx^2 + C,$  provided that $C$ is nonsingular. Furthermore, \texttt{IntegralQuarticPoints([a,b,c,d,e])} provides the integral points 
on the curve $C : y^2 = ax^4 + bx^3 +cx^2+dx+e.$
\end{remark}

\subsection{Examples}\ \\
We assume that functions for computing integer points in elliptic curves given by Sagemath and magma do not miss any integer point. \\
$({\bf i}).$ $f(x)=x^3+x+1$ and $g(x)=x^4+2x^3-3x^2+4x+4.$
This is example 9.2 from \cite[Section 9]{BBM} which was solved using the heavy machinery of jacobian varieties. It 
was proved that the integer points of curve
$$y^2=(x^3+x+1)(x^4+2x^3-3x^2+4x+4)$$
with $y>0,$ are
$$\Sigma=\{(-1,2),(0,2),(-2,12),(3,62)\}.$$
This specific curve has genus $3$ and the rank of its jacobian is again $3$ so the method of \cite{BBM} can, in principal,  be applied. Of course, our method can be also considered {\it heavy}, since we use elliptic logarithm method for computing the integer points of many elliptic curves (here only $4$). In our opinion, our method is more simple and practical since we already have efficient implementations (for instance in Magma and Sagemath) for calculating integer points on elliptic curves.
 Since $dy^2=f(x)$ and $dy^2=g(x)$ both define elliptic curves, we work with the cubic polynomial, $dy^2=f(x).$ In sagemath, we compute in some seconds that indeed the set of integer points of $y^2=f(x)g(x)$  is the set $\Sigma.$\\
$({\bf ii}).$ The following example is from \cite[example 8.1]{BruinStoll}, where they studied its rational points.
Let $$C:y^2=	2x^6 + x^4 + 3x^2 - 2.$$ This can be rewritten as $y^2=(x^4 + x^2 + 2)(2x^2 - 1).$ We set $f(x)=x^4+x^2+2$ and $g(x)=2x^2-1.$ We work with the curve $C_d'$ (see, \ref{C'}). Executing in magma 
\texttt{SIntegralLjunggrenPoints([d,1,1,2],[])}, for $d\in\{\pm 1,\pm 11\},$ we get only for $d=1$ the integer points $(\pm 1,2)$ (restricting $y$ to positive integers). Substituting to $f(x)g(x)$ we get the points $(\pm 1,4).$ So, $$C({\mathbb{Z}})=\{(\pm 1,\pm 4) \}.$$
$({\bf iii}).$ $C:y^2=x(x^2 - 1)(x^2 - 1600).$ We set 
 $f(x)=x(x^2-1600)$ and $g(x)=x^2 - 1.$ Then, the integer solutions with $y>0	,$ are $$\{(-25,3900) \}.$$
$({\bf iv}).$ $C:y^2=x(x^2 - 1)(x - 4)(x - 5).$ We set 
 $f(x)=x(x-1)(x-5)$ and $g(x)=(x + 1)(x - 4).$ Then, the integer solutions with $y>0,$ are 
$$\{(2,6),(9,120) \}.$$
$({\bf v}).$ $f(x)=x(x-1)(x-2)$ and $g(x)=(x + 40)(x + 4)(x - 4)(x - 7).$ Then the solutions of $y^2=f(x)g(x)$ with $y>0,$ are
$$\{(-7,2772),(16,20160) \}.$$
\subsection{The Diophantine equations $y^2=(x^4\pm 1)g(x)$}
We start our study with $y^2=(x^4-1)g(x).$
\begin{proposition}\label{prop:x^4-1}
Let $C:y^2=(x^4-1)g(x)$ for $g(x)\in {\mathbb{Z}}[x]$ with $g(\pm 1)$ and $g(\pm i)$ not zero. Set $C_+({\mathbb{Z}}) =\{(a,b)\in C({\mathbb{Z}}):b>0\}.$ Then,
$$|C_+({\mathbb{Z}})|\leq 2^{\omega(N)+2},$$
where $N=g(1)g(-1)g(i)g(-i)$ and $\omega(N)$ is the number of prime divisors of $N.$
\end{proposition}
\begin{proof}
Set $f(x)=x^4-1.$ The twists here are of the form $dy^2=x^4-1,$ where $d|R_0,$ $R_0=res_x(f(x),g(x)).$ It is easy to see that $R_0=N.$ From \cite{Ljunggren} (also see \cite{Cohn}) the Diophantine equation $dy^2=x^4-1$ has at most two solutions in positive integers for $d\ge 2.$ Also, from \cite[Lemma 1.1]{Samuel}, for $d=2$ we have only the trivial solution $(\pm 1,0)$. For $d=1$ and $d<0$ we have $y=0.$ Since we are interested in positive $y$, we ignore the cases $d\leq 1.$  Now, each solution $(a,b)$ (we have at most two) of $dy^2=x^4-1$, $d\ge 2$, it might provide a solution on $C$ for the two possible signs of $a.$ I.e. we have to check if $(a^4-1)g(a)$ or $(a^4-1)g(-a)$ is square. Thus, $|C_+({\mathbb{Z}})|$ is bounded above by the number of equations $dy^2=x^4-1,$ with $d\ge 2,$  squarefree divisor of $N,$ and the result is multiplied by $4.$ So, $|C_+({\mathbb{Z}})|<4\cdot 2^{\omega(n)}.$ The result follows.
\end{proof}
\begin{remark}\label{rem:1}
Since, for large $N$ we have 
$\omega(N)\sim \frac{\log{N}}{\log{\log{N}}}$ we get that for every $\varepsilon>0,$
$2^{\omega(N)}\ll_{\varepsilon} N^{\varepsilon}$ thus,
$$|C_{+}({\mathbb{Z}})|\ll_{\varepsilon} N^{\varepsilon}.$$
With $\ll_{\varepsilon} N^{\varepsilon}$ we mean that there is a constant depending on $\varepsilon$, say $c(\varepsilon)$, such that $<c(\varepsilon)N^{\varepsilon}.$

\end{remark}
Furthermore, it is easy to compute all the integer solutions of $y^2=(x^4-1)g(x)$ using \cite{Cohn}. First, we compute $N=g(1)g(-1)g(i)g(-i)$ and we construct the set $\Sigma_N=\{d\in {\mathbb{Z}}_{>0} : d|N\ \text{and}\ d\ \text{is\ squarefree}\}.$ For every integer $d\in \Sigma_N$ we compute the fundamental solution of the Pell equation $X^2-dY^2=1$ say $u+v\sqrt{d}.$ Then we compute $a=\sqrt{u}$ and $a=\sqrt{2u^2-1}$ (only for $d=1785$ both the square roots are integers). If it occurs $a$ to be integer, then we check if $(a^4-1)g(\pm a)$ is a square. If it is, then we found an integer solution $(a,\sqrt{(a^4-1)g(a)})$ or $(a,\sqrt{(a^4-1)g(-a)}).$ In appendix we provide a pseudocode and we compute the integer points of the three families 
$$y^2=(x^4-1)(x^2+kx+1),\ y^2=(x^4-1)(x^3+kx^2+1),\ y^2=(x^4-1)(x^3+kx+1)$$
for various values of $k.$

We can also study the {\it{symmetric}} Diophantine equation $y^2=(x^4+1)g(x).$ Set $\zeta_j = \frac{1}{\sqrt{2}}(\pm1 \pm i),$ $j=1,2,3,4$ the quartic roots of $-1$ (take all the possible signs, and it is not important here how you pick $\zeta_j$'s). Now the analogous Proposition of \ref{prop:x^4-1} is the following.
\begin{proposition}
Let $W:y^2=(x^4+1)g(x)$ for $g(x)\in {\mathbb{Z}}[x]$ with $g(\zeta_j)\not=0\ (j=1,2,3,4).$ Set $W_+({\mathbb{Z}}) =\{(a,b)\in W({\mathbb{Z}}):b>0\}.$ Then,
$$|W_+({\mathbb{Z}})|\leq 2^{\omega(N)+1},$$
where $N=\prod_{j=1}^{4}g(\zeta_j)$ and $\omega(N)$ the number of prime divisors of $N.$
\end{proposition}
\begin{proof}
Set $f(x)=x^4-1.$ The twists here are of the form $dy^2=x^4+1,$ where $d|R_0$ with $R_0=res_x(f(x),g(x)).$ It is easy to see that $R_0=N.$ From \cite{Ljunggren2} (see also \cite{Cohn_negative_pell}), the Diophantine equation $dy^2=x^4+1$ has at most one solution in positive integers for $d\ge 2.$ 
For $d=1$ we have $x=0$ and for $d<0$ we do not have any integer solution. Now, each solution $(a,b)$ (we have at most one) of $dy^2=x^4+1$ it might provide a solution on $W$ for the two possible signs of $a.$ I.e. we have to check if $(a^4+1)g(a)$ or $(a^4+1)g(-a)$ is square. Thus,
 $$|W_+({\mathbb{Z}})|\le 2|\{d\in{\mathbb{Z}}_{>0}:d|N,\ d\text{ squarefree} \}|=2\cdot 2^{\omega(N)}.$$  So the result follows.
\end{proof}
\begin{remark}\label{rem:2}
As in Remark \ref{rem:1}, for every $\varepsilon>0,$
$$|W_{+}({\mathbb{Z}})|\ll_{\varepsilon} N^{\varepsilon}.$$
\end{remark}
In order to solve the Diophantine equation $dy^2=x^4+1$ we use Cohn's algorithm in \cite{Cohn_negative_pell}. To apply this algorithm, we first compute a fundamental solution of $V^2-dU^2=-1$ (if it exists, see the remark below), say $(v_0,u_0).$ Also, say $A$ be the square free part of $v_0.$ We set, $\mu = v_0+u_0\sqrt{d}$ and $\lambda=v_0-u_0\sqrt{d}$ and $s_n =\frac{\mu^n+\lambda^n}{2}.$ Finally, compute $s_A.$ Then, $x=\pm\sqrt{s_A}.$

\begin{remark}
The negative Pell equation $Y^2-NX^2=-1$ is solvable if and only if the period of the continued fraction of $\sqrt{N}$ is odd (see \cite[Chapter 8, Section 5, Theorem 9]{Sierpinski}).
\end{remark}

\begin{corollary}
Let $C_n:y^2=(x^4+1)(x^n+1),$ then
$$C_n({\mathbb{Z}})\subset\{(-1,0),(\pm 1, \pm 2), (0,\pm 1)\},$$ for all $n\not\equiv 0 \pmod 4.$
\end{corollary}
\begin{proof}
Suppose that $x\not=0.$ Straightforward calculations show that $N$ is a power of $2$. In fact $res_x(f,g)=2^{i}, i\in\{1,2\},$ the exponent $i=2$ appears when $i\equiv 2\pmod{4}$. So all the squarefree divisors of $N$ are $\{-2,2,-1,1\}$. The negative Pell equation for $d=2$  provides the fundamental solution $(1,1),$ and so $x=\pm 1.$  For $d\in\{1,-1,-2\},$ the corresponding twists do not provide any non trivial integer solution. The result follows.
\end{proof}

\section{Conclusion}\label{sec:conclusion}
Let 
\begin{equation}\label{hyperelliptic_equation}
y^2=F(x),\ F\in {\mathbb{Z}}[x],\ \deg{F}>3.
\end{equation}
We assume that the polynomial $F(x)$ in (\ref{hyperelliptic_equation}) is reducible polynomial with one irreducible component a cubic  polynomial or quartic polynomial. For instance, if $F(x)$ has three integer roots the previous constraint is satisfied. 
We reduced the study of Diophantine hyperelliptic equations, with the previous constraint, to the study of integer points on elliptic curves. Since, the integer points of elliptic curves have been studied very thoroughly, it is easy to find the integer points of (\ref{hyperelliptic_equation}).

As far as the superellipitc equations ${\mathcal{S}}:y^p=f(x)g(x),$ $p\ge 3$ prime, we provide also an unramified map $\Phi : Y_d \rightarrow {\mathcal{S}}$ for known $d$'s such that 
$${\mathcal{S}}({\mathbb{Z}})\subset \Phi\Big(\bigcup	_{d|N}Y_d({\mathbb{Z}})\Big),$$
where $N$ is $p-$free integer equal to the resultant of $f,g.$
Folllowing similar ideas as in the hyperellipitc case we study curves of the form $y^p=(Ax^p+B)g(x).$

\ \\
{\sc{\Large{Appendix}}}\\
{\bf A. Pseudocode for the diophantine equation $y^2=(x^4-1)g(x)$.}\\
Now, we easily get the following pseudocode\footnote{for an implementation in Sagemath see \\ \url{https://github.com/drazioti/hyper_super_elliptic/blob/main/special_case_hyperelliptic_1.sage}}.\\\\
{\footnotesize{
\noindent
{\bf Input}: $g(x)\in{\mathbb{Z}}[x]$ ($g(\pm1)\not=0,\ g(\pm i)\not=0$)
\ \\
{\bf Output}: The integer solutions $(x,y)$ of $y^2=(x^4-1)g(x)\ (y>0).$ \\\\
    \texttt{01.} \texttt{$L$ an empty list}\\
	\texttt{02.} \texttt{$N\leftarrow g(1)g(-1)g(i)g(-i)$}\\
    \texttt{03.} \texttt{$SF\leftarrow \{d\in {\mathbb{Z}}_{>0}:d|N,\ d\ \text{is\ square free}\}$\\
	\texttt{04.} {\bf for} $d$ in $SF$\\
	\texttt{05.} \hspace{0.3cm}  \texttt{Compute the fundamental solution $u+v\sqrt{d}$ of $X^2-dY^2=1$}\\
	\texttt{06.} \hspace{0.3cm}  {\bf if} \texttt{$a=\sqrt{u}$ is integer}\\
	\texttt{07.} \hspace{0.6cm} {\bf if} \texttt{$g(a)(a^4-1)$ is square say $b^2$ $(b>0)$}\\
	\texttt{08.} \hspace{1cm}  Append list $L$ with $(a,b)$\\
	\texttt{09.} \hspace{0.6cm} {\bf if} \texttt{$g(-a)(a^4-1)$ is square say $b^2$ $(b>0)$}\\
	\texttt{10.} \hspace{1cm}  Append list $L$ with $(-a,b)$\\
	\texttt{11.} \hspace{0.3cm}  {\bf if} \texttt{$a=\sqrt{2u^2-1}$ is integer}\\
	\texttt{12.} \hspace{0.6cm} {\bf if} \texttt{$g(a)(a^4-1)$ is square say $b^2$ $(b>0)$}\\
	\texttt{13.} \hspace{1cm}  Append list $L$ with $(a,b)$\\
	\texttt{14.} \hspace{0.6cm} {\bf if} \texttt{$g(-a)(a^4-1)$ is square say $b^2$ $(b>0)$}\\
	\texttt{15.} \hspace{1cm}  Append list $L$ with $(-a,b)$\\
	\texttt{16.}    return $L$
}}}
\ \\
\begin{center}

\begin{tabular}{ |p{1cm}||p{5cm}|  }
 \hline
 \multicolumn{2}{|c|}{$y^2=(x^4-1)(x^2+kx+1)$, $1< k\le1000$} \\
 \hline
 $k$ & Integer points $(x,y)$, $y>0$\\
 \hline
 5   & $(2,15)$   \\
 26  &  $(5,312)$  \\
 65  & $(2,45), (5,468)$ \\
 101 & $(10,3333)$\\
 185 & $(2,75)$ \\
 290 & $(17,20880)$ \\
 365 & $(2,105)$ \\
 377 & $(5,1092)$ \\
 494 & $(5,1248)$ \\
 605 & $(2,135)$ \\
 677 & $(26, 91395)$ \\
 905 & $(2,165)$ \\
  \hline
\end{tabular}

\end{center}
\begin{center}
\begin{tabular}{ |p{1cm}||p{5cm}|  }
 \hline
 \multicolumn{2}{|c|}{$y^2=(x^4-1)(x^3+kx^2+1)$, $1< k\le500$} \\
 \hline
 $k$ & Integer points $(x,y)$, $y>0$\\
 \hline
 9   & $(5,468)$   \\
 19  &  $(-5,468)$  \\
  \hline
\end{tabular}
\end{center}
For the second Table, for some values of $k$, for instance $k=421$, the algorithm is memory {\it{expensive}}. I.e. the fundamental solution of Pell equation is very large. So, you may need enough memory to construct the second table. You will need enough hours for all $k$'s.

\begin{center}
\begin{tabular}{ |p{1cm}||p{5cm}|  }
 \hline
 \multicolumn{2}{|c|}{$y^2=(x^4-1)(x^3+kx+1)$, $1< k\le200$} \\
 \hline
 $k$ & Integer points $(x,y)$, $y>0$\\
 \hline
 3   & $(2,15)$   \\
 6  &  $(5,312)$  \\
 11  &  $(10,3333)$  \\
 18  &  $(17,20880)$  \\
 27  &  $(21,91395)$  \\
 38  &  $(37, 312360)$  \\
 45  &  $(5,468)$  \\
 51  &  $(50,892857)$  \\
 63  &  $(2,45)$  \\ 
 66  &  $(65,2231328)$  \\
 83  &  $(82,5023575)$  \\
 102  &  $(82,5023575)$  \\
 123  &  $(122, 20139405)$  \\
 146  &  $(145, 36837552)$  \\
 168  &  $(41, 462840])$  \\
 171  &  $(170, 64246923)$  \\
 183  &  $(2,75)$  \\
 197  &  $(197, 107581320)$  \\
  \hline
\end{tabular}
\end{center}
We remark that, if $k=n^2+2,$ then we always have the solution 
$$(n^2+1,(n^4+2n^2+2)(n^2+2)n).$$

\end{document}